\documentclass{amsart}

\usepackage[backend=bibtex, style=alphabetic, sorting=nty, maxnames=10]{biblatex}
\usepackage{amssymb,latexsym,amsmath,amsthm,amsfonts, enumerate}
\usepackage{color}
\usepackage[all]{xy}
\usepackage{tikz}
\usepackage{tikz-cd}
\usepackage{tabu}

\usetikzlibrary{positioning,shapes,shadows,arrows,snakes}

\usepackage[colorlinks=true,hyperindex]{hyperref}
\newcommand\myshade{85}
\colorlet{mylinkcolor}{violet}
\colorlet{mycitecolor}{blue}
\colorlet{myurlcolor}{cyan}

\hypersetup{
  linkcolor  = mylinkcolor!\myshade!black,
  citecolor  = mycitecolor!\myshade!black,
  urlcolor   = myurlcolor!\myshade!black,
  colorlinks = true,
}

\numberwithin{equation}{section}

\newtheorem{theorem}{Theorem}[section]
\newtheorem{proposition}[theorem]{Proposition}

\newtheorem{lemma}[theorem]{Lemma}
\newtheorem{theoremA}{Theorem}

\theoremstyle{definition}
\newtheorem{remark}[theorem]{Remark}
\newtheorem{example}[theorem]{Example}
\newtheorem{definition}[theorem]{Definition}

\newcommand{\ZZ}{\mathbb{Z}}
\newcommand{\xx}{\mathbf{x}}

\newcommand{\Hom}{\mathrm{Hom}}
\newcommand{\Ext}{\mathrm{Ext}}

\newcommand{\ssi}{\Leftrightarrow}

\newcommand{\zg}{\gamma}
\newcommand{\zG}{\Gamma}
\newcommand{\zs}{\sigma}

\newcommand{\calf}{\mathcal{F}}
\newcommand{\cala}{\mathcal{A}}
\newcommand{\calc}{\mathcal{C}}
\newcommand{\call}{\mathcal{L}}
\newcommand{\cald}{\mathcal{D}}

\newcommand{\zehn}{\hspace{10pt}}
\newcommand{\fuenf}{\hspace{5pt}}
\newcommand{\fuenfm}{\hspace{-5pt}}

\title{Frieze vectors and unitary friezes}

\author{Emily Gunawan}
\address{Department of Mathematics\\ University of Oklahoma\\ 
Norman, OK 73019-3103, USA}
\email{egunawan@ou.edu}

\author{Ralf Schiffler}
\address{Department of Mathematics\\ University of Connecticut\\
Storrs, CT 06269-1009, USA}
\thanks{The authors were supported by the NSF-CAREER grant  DMS-1254567 and by the University of Connecticut. The second author was also supported by the NSF grant  DMS-1800860}
\email{schiffler@math.uconn.edu}


\begin{document}

\begin{abstract}
Let $Q$ be a quiver without loops and 2-cycles, let $\cala(Q)$ be the corresponding cluster algebra and let $\xx$ be a cluster. 
 We introduce a new class of integer vectors which we call frieze vectors relative to $\xx$. These frieze vectors are defined as solutions of certain  Diophantine equations given by the cluster variables in the cluster algebra. We show that every cluster gives rise to a frieze vector and that the frieze vector determines the cluster. 

We also study friezes of type $Q$ as homomorphisms from the cluster algebra to an arbitrary integral domain. Moreover, we show that every positive integral frieze of affine Dynkin type $\widetilde{\mathbb{A}}_{p,q}$ is unitary, which means it is obtained by specializing each cluster variable in one cluster to  the constant 1. This completes the answer to the question of unitarity for all positive integral friezes of Dynkin and affine Dynkin types. 
  
 \end{abstract}

\maketitle

\setcounter{tocdepth}{1}
{\tableofcontents
}
\section{Introduction}\label{sect 1}

Let $Q$ be a quiver without loops and 2-cycles and let $\cala(Q)$ be the corresponding cluster algebra with trivial coefficients. We define a {\em frieze of type $Q$} to be a ring homomorphism $\calf\colon\cala(Q)\to R$ from the cluster algebra to an integral domain $R$. The frieze $\calf$ is called \emph{non-zero} if every cluster variable is mapped to a non-zero element of $R$ and $\calf$ is said to be \emph{unitary} if there exists a cluster $\xx$ such that  $\calf(x)$ is a unit in $R$, for all $x\in\xx$. Moreover $\calf$ is called \emph{integral} if $R=\ZZ$, and \emph{positive} if $R=\ZZ$ and every cluster variable is mapped to a positive integer. 

Positive integral friezes of Dynkin type $\mathbb{A}_n$ are precisely the classical Conway-Coxeter friezes, where the classical frieze pattern is given by displaying the values of $\calf$ on the cluster variables in the shape of the Auslander-Reiten quiver of the cluster category. 

Every non-zero frieze is determined by its values $\calf(\xx)=(a_1,\ldots,a_n)$ on an arbitrary cluster $\xx=(x_1,\ldots,x_n)$ in $\cala(Q)$. It is therefore natural to ask which values $(a_1,\ldots,a_n)$ produce positive unitary integral friezes. We call such a vector $(a_1,\ldots,a_n)$ a \emph{unitary frieze vector relative to the cluster $\xx$}. Our first main result is the following.

\begin{theoremA}\label{thm A}
  Let $Q$ be a quiver without loops and 2-cycles and let $\xx=(x_1,\ldots,x_n)$ be an arbitrary cluster of $\cala(Q)$. Then there is a bijection
 \[
\begin{array}{rcl}
 \phi\colon\{\textup{unordered clusters in $\cala(Q)$}\} &\longrightarrow&\left\{\begin{array}{l}\textup{positive unitary frieze } \\ \textup{vectors relative to $\xx$}\end{array}\right\} \\
 \xx'=\{x_1',\ldots,x_n'\} &\longmapsto& \phi(\xx')=(a_1,\ldots,a_n).
\end{array}\]
\end{theoremA}
Thus every cluster $\xx'$ defines a unique unitary frieze vector. One can thus think of the frieze vectors as another parametrization of the clusters in the cluster algebra. The frieze vectors are different from other known vectors  appearing in cluster algebra theory like denominator vectors, $c$-vectors or $g$-vectors.
\smallskip

Our second main result is about the unitarity of positive integral friezes. Since Conway and Coxeter's work in 1973, it is known that every positive integral frieze of Dynkin type $\mathbb{A}$ is unitary. For Dynkin types $\mathbb{D}$ and $\mathbb{E}$ there exist non-unitary positive integral friezes, see \cite{FP}. We extend these results to the affine Dynkin types as follows.

\begin{theoremA} \label{thm B}
 Let $Q$ be a quiver of type $\widetilde{\mathbb{A}}_{p,q}$ and let $\calf\colon\cala(Q)\to \ZZ$ be a positive integral frieze. Then $\calf$ is unitary.
\end{theoremA}

Our proof is constructive. We give an algorithm that starts from an arbitrary positive integral frieze $\calf$ and produces the unique cluster $\xx$ such that $\calf(\xx)=(1,\ldots,1)$.
In the other affine types $\widetilde{\mathbb{D}}$ and $\widetilde{\mathbb{E}}$, there are non-unitary positive integral friezes. 

It is natural to ask if friezes of types $\mathbb{A}$ and $\widetilde{\mathbb{A}}$ remain unitary if one replaces the ring of integers by other integral domains. However, already over the Gaussian integers we  give an example of a non-unitary frieze of Dynkin type 
 $\mathbb{A}_2$. The classification of friezes over the Gaussian integers or other integral domains besides $\mathbb{Z}$ is open even in type $\mathbb{A}$. For type $\mathbb{A}_1$ there are 12 non-zero friezes over the Gaussian integers, see \cite{F}.

The paper is organized as follows. In section \ref{sect 2}, we give the formal definition of friezes and show how they are a generalization of Conway-Coxeter friezes. We also give several examples of friezes of type $\mathbb{A}_3$ over different rings. Section \ref{sect 3} is devoted to the definition of frieze vectors and the proof of Theorem \ref{thm A}, and Theorem \ref{thm B} is proved in section \ref{sect 4}.

\section{Friezes}\label{sect 2}
Friezes of type $\mathbb{A}_n$ were classified by  Conway and Coxeter in \cite{ConCox1,ConCox2} in 1973. More than 30 years later, Caldero and Chapoton discovered a relation between friezes and cluster algebras in \cite{CC}. Since then friezes were studied by many authors, see for example \cite{BM, ARS, KS, MG1, MGOT, FP, BFGST,  BRM, BFPT, GMV, LLMSS}.
For a survey we refer the reader to \cite{MG2}.

Usually classical friezes are defined as certain planar arrays of positive integers that satisfy a diamond relation. In this paper however, we take a different point of view and we define a frieze to be a homomorphism from an arbitrary cluster algebra to an arbitrary integral domain $R$.  The usual planar array is obtained from the Auslander-Reiten quiver of the corresponding cluster category by replacing the indecomposable objects (i.e. the vertices of the Auslander-Reiten quiver) by the values of the homomorphism on the corresponding cluster algebra elements. Friezes as homomorphisms to the integers were also considered in \cite{F,FP} and in \cite[Appendix B]{BFGST2}, and friezes with values in subsets of the complex numbers in \cite{CH}.

\subsection{Definition}
Let $Q$ be a quiver without loops and 2-cycles and let $\cala(Q)$ be the corresponding cluster algebra with trivial coefficients, see \cite{FZ}. We could just as well include coefficients in our definition, but since we are not using them in this paper we impose trivial coefficients for simplicity.

\begin{definition}
 (1) A \emph{frieze of type $Q$} is a ring homomorphism
 \[\calf\colon \cala(Q)\longrightarrow R\] 
 from the cluster algebra to an integral domain $R$. The frieze is called \emph{integral} if $R=\mathbb{Z}$.
 
 (2) A frieze $\calf\colon \cala(Q)\longrightarrow R$ is said to be \emph{unitary} if there exists a cluster $\xx$ in $\cala(Q)$ such that every cluster variable $x\in\xx$ is mapped by $\calf$ to a unit in $R$.
 
 (3) A frieze is said to be \emph{non-zero} if every cluster variable in $\cala(Q)$ is mapped by $\calf$ to a non-zero element of $R$.
 
 (4) An integral frieze is said to be \emph{positive}  if every cluster variable in $\cala(Q)$ is mapped by $\calf$ to a positive integer.
\end{definition}

\begin{remark}
 Our definition of unitary friezes agrees  with that of \cite{MG1,FP} for positive integral friezes. Note however that if the integral frieze is not positive, we also allow specialization at -1.
\end{remark}

\subsection{Cluster category and Auslander-Reiten quiver} Let $Q$ be a quiver without loops and 2-cycles.
If the quiver $Q$ is mutation equivalent to an acyclic quiver $Q'$, we let $\calc$ 
be the cluster category 
$\calc_Q=\cald^b(\textup{mod}\,kQ')/\tau^{-1}[1]$ introduced in \cite{BMRRT} and in \cite{CCS} for type $\mathbb{A}$. More generally, if $Q$ comes with a non-degenerate potential, we let $\calc$ 
be the generalized cluster category introduced in \cite{A}. We denote by 
$\zG(\calc)$  
the Auslander-Reiten quiver of 
$\calc$. 
Its vertices are the isoclasses of indecomposable objects in 
$\calc$ 
and its arrows are given by irreducible morphisms in 
$\calc$. 
If $Q$ is mutation equivalent to an acyclic quiver $Q'$, then 
$\zG(\calc)$ 
has a special connected component, called the \emph{transjective component}, that contains both the preprojective component and the preinjective component of $\textup{mod}\,kQ'$. In finite type, this transjective component is all of 
$\zG(\calc)$.

The cluster category is a triangulated category  equipped with a Serre functor (if it is Hom-finite) given by the Auslander-Reiten translation $\tau$. Moreover 
$\calc$
has Auslander-Reiten triangles and it is 2-Calabi-Yau, meaning that 
$\Ext^1_{\calc}(X,Y)\cong D\Ext^1_{\calc}(Y, X)$, 
where $D=\Hom(-,k)$ denotes the standard duality, see \cite{K,A}. An object 
$X\in \calc$ 
is called rigid if 
$\Ext^1_{\calc}(X,X)=0$, 
and an indecomposable rigid object in 
$\calc$
is called reachable if it can be reached under mutation from the initial cluster-tilting object. If $Q$ is mutation equivalent to an acyclic quiver all rigid indecomposable objects are reachable and all indecomposables in the transjective component are rigid. 

The cluster character is a map 
$X_?\colon\calc\to \textup{Frac} \cala(Q)$ 
from the cluster category to the field of fractions of the cluster algebra that maps (isoclasses of reachable) indecomposable rigid objects in 
$\calc$
bijectively to cluster variables in $\cala(Q)$, see \cite{CC,CK,CK2,Palu,CKLP,FK}. The key for the relation to classical friezes lies in the image of  Auslander-Reiten triangles under the cluster character. This is expressed in the following proposition, which is a special case of 
\cite[Proposition~2.3(a)]{DG}. For convenience of the reader, we include a proof here. 

\begin{proposition}\label{prop 2.2} Let $Q$ be an acyclic quiver.
 If $\tau N \to \oplus_{i\in I} M_i \to N\to \tau N[1]$ is an Auslander-Reiten triangle in the transjective component of $\calc_Q$ with $\tau N, M_i, N$ indecomposable rigid objects then we have the following identity in the cluster algebra.
 \[ X_{\tau N}\, X_{N} = \prod_{i\in I} X_{M_i} +1.\]
\end{proposition}
\begin{proof} Since $N$ is rigid transjective, we have $\dim \Ext^1(N,\tau N)=1$ and therefore $N$ and $\tau N$ form an exchange pair
 \cite[Theorem 7.5]{BMRRT}. This implies that there are unique (up to isomorphism) triangles 
 \[\tau N \to \oplus_{i\in I} M_i \to N\to \tau N[1]\quad  \textup{and} \quad N \to \oplus_{i\in I'} M'_i \to \tau N\to  N[1]\] such that $X_{\tau N} \,X_N = \prod_{i\in I} X_{M_i} +\prod_{i\in I'} X_{M'_i}$. 
 Now, in the cluster category, we have $\tau=[1]$, and thus the second triangle is isomorphic to  $N \to 0\to  N[1]\stackrel{1}{\to}  N[1]$. This completes the proof.
\end{proof}
 
\begin{remark}
(1) This proposition gives the so-called diamond relation in the friezes.

(2) If we were considering cluster algebras with non-trivial coefficients the constant 1 on the right hand side of the equation in Proposition \ref{prop 2.2} would be replaced by a coefficient monomial. Friezes of that type were studied in \cite{BRM}.
\end{remark}

\subsection{Examples}\label{sect 2.3}
(1) The identity homomorphism $\cala(Q)\to\cala(Q)$ is a non-zero frieze of type $Q$.  For example, if $Q$ is the type $\mathbb{A}_3$ quiver $1\to 2\leftarrow 3$, we can visualize this frieze in the Auslander-Reiten quiver of $\calc_Q$ as follows.
First let us write down the Auslander-Reiten quiver. 

\[\xymatrix@!@R10pt@C10pt{ 
& {\begin{smallmatrix} 3\\2 \end{smallmatrix}}{\scriptstyle[1]} \ar[rd] &&
 {\begin{smallmatrix} 3\\2 \end{smallmatrix}} \ar[rd] &&
 {\begin{smallmatrix} 1 \end{smallmatrix}} \ar[rd] &&
 {\begin{smallmatrix} 1\\2 \end{smallmatrix}}{\scriptstyle[1]} \\
 {\begin{smallmatrix} 2 \end{smallmatrix}}{\scriptstyle[1]} \ar[rd]\ar[ru] &&
 {\begin{smallmatrix} 2 \end{smallmatrix}} \ar[rd]\ar[ru] &&
 {\begin{smallmatrix} 1\ 3\\2 \end{smallmatrix}} \ar[rd]\ar[ru] &&
 {\begin{smallmatrix} 2 \end{smallmatrix}}{\scriptstyle[1]} \ar[rd]\ar[ru] &&
\\
& {\begin{smallmatrix} 1\\2 \end{smallmatrix}}{\scriptstyle[1]} \ar[ru] &&
 {\begin{smallmatrix} 1\\2 \end{smallmatrix}} \ar[ru] &&
 {\begin{smallmatrix} 3 \end{smallmatrix}} \ar[ru] &&
 {\begin{smallmatrix} 3\\2 \end{smallmatrix}}{\scriptstyle[1]} \\
}
\] 
Here we use a standard notation for the representations of the quiver $Q$, see for example \cite{Schiffler}, and $[1]$ denotes the shift. Vertices with the same label are identified, so the quiver lies on a Moebius strip.  The Auslander-Reiten translation $\tau$ is the horizontal translation to the left. For example $\tau {\begin{smallmatrix} 3 \end{smallmatrix}} ={\begin{smallmatrix} 1\\2 \end{smallmatrix}}$. The Auslander-Reiten triangles are given by the meshes in the Auslander-Reiten quiver, for example
\[\to {\begin{smallmatrix} 1\\2 \end{smallmatrix}}{\scriptstyle[1]}  \to {\begin{smallmatrix} 2 \end{smallmatrix}} \to {\begin{smallmatrix} 1\\2 \end{smallmatrix}} \to 
\qquad \textup{and} \qquad 
\to {\begin{smallmatrix} 2 \end{smallmatrix}} \to {\begin{smallmatrix} 1\\2 \end{smallmatrix}} \oplus {\begin{smallmatrix} 3\\2\end{smallmatrix}} \to {\begin{smallmatrix} 1\ 3\\2 \end{smallmatrix}} \to 
\]
are Auslander-Reiten triangles.

The identity homomorphism $\cala(Q)\to\cala(Q)$ gives the following frieze.
\[\scalebox{0.9}{
\xymatrix@R-10pt@C-10pt{ 
&\zehn x_3\zehn  \ar[rd] &&
 \frac{x_1x_3+1+x_2}{x_2x_3} \ar[rd] &&
 \frac{x_2+1}{x_1} \ar[rd] &&
 \zehn x_1
 \zehn \\
\fuenf x_2\fuenf \ar[rd]\ar[ru] &&
 \frac{x_1x_3+1}{x_2} \ar[rd]\ar[ru] &&
\fuenfm \fuenfm \frac{x_2^2+2x_2+1+x_1x_3}{x_1x_2x_3}\fuenfm\fuenfm \ar[rd]\ar[ru] &&
 \fuenf x_2\fuenf  \ar[rd]\ar[ru] &&
\\
& x_1 \ar[ru] &&
  \frac{x_1x_3+1+x_2}{x_1x_2} \ar[ru] &&
  \frac{x_2+1}{x_3} \ar[ru] &&
 x_3 \\
} 
}
\]

This is an example of a non-zero frieze of type $\mathbb{A}_3$. Notice that the Auslander-Reiten triangles give the usual diamond rules, for example 
\[x_1 
  \ \frac{x_1x_3+1+x_2}{x_1x_2} = \frac{x_1x_3+1}{x_2} \ +\ 1 
\qquad \textup{and} \qquad \]
\[\frac{x_1x_3+1}{x_2}\ \frac{x_2^2+2x_2+1+x_1x_3}{x_1x_2x_3}  
\ =\  \frac{x_1x_3+1+x_2}{x_1x_2} \  \frac{x_1x_3+1+x_2}{x_2x_3}  \ + \ 1
\]

\noindent (2) Specializations. We compute several specializations of the example above.

(i) Specializing $x_1=x_2=x_3=1$, we obtain the following unitary positive integral frieze. 
\[\xymatrix{ 
& 1 \ar[rd] &&
3 \ar[rd] &&
2 \ar[rd] &&
1  &&
\\
1 \ar[rd]\ar[ru] &&
2 \ar[rd]\ar[ru] &&
5 \ar[rd]\ar[ru] &&
1 \ar[rd]\ar[ru] &&
 \\
&
1\ar[ru] &&
3\ar[ru] &&
2\ar[ru] &&
1 &&
}
\] 
Here the previous examples of the diamond rules become simply
\[1\cdot 3 = 2+1 \qquad \textup{and} \qquad 2\cdot 5 =3\cdot 3+1.\]
This is an example of a classical  Conway-Coxeter frieze; let us point out that one can extend this frieze pattern by a row of 1's above and below the current pattern, which is how the Conway-Coxeter friezes are usually represented. We will not include these rows of 1's in this article. 

(ii) Specializing $x_1=x_2=1$ and $x_3=-1$, we obtain the following unitary  integral frieze which is non-positive, not even non-zero. 
\[\xymatrix{ 
& -1 \ar[rd] &&
-1\ar[rd] &&
2 \ar[rd] &&
1  &&
\\
1 \ar[rd]\ar[ru] &&
0 \ar[rd]\ar[ru] &&
-3 \ar[rd]\ar[ru] &&
1 \ar[rd]\ar[ru] &&
 \\
&
1\ar[ru] &&
1\ar[ru] &&
-2\ar[ru] &&
-1 &&
}
\]
Our example diamond relations become here $1\cdot 1 =0+1$ and $0\cdot(-3)=(-1)\cdot 1 +1$.

(iii) Specializing $x_1=1$, $x_2=i$, and $x_3=i$, we obtain the following unitary non-zero frieze in the Gaussian integers $\mathbb{Z}[i]$. 
\[\xymatrix{ 
& \fuenf i \fuenf \ar[rd] &&
 \fuenfm-1-2i \fuenfm \ar[rd] &&
1+i \ar[rd] &&
 \fuenf1  \fuenf &&
\\
i \ar[rd]\ar[ru] &&
1-i \ar[rd]\ar[ru] &&
-3i \ar[rd]\ar[ru] &&
i \ar[rd]\ar[ru] &&
 \\
&
1\ar[ru] &&
2-i\ar[ru] &&
1-i\ar[ru] &&
i &&
}
\] 
Here our example diamond relations become $1\cdot (2-i) =(1-i)+1$ and $(1-i)\cdot(-3i)=(-1-2i)\cdot (2-i) +1$.

(iv) Specializing $x_1=1$, $x_2=\frac{1+\sqrt{-3}}{2}$, $x_3=1$, we obtain the following unitary non-zero frieze in the quadratic integer ring $\mathbb{Z}[\sqrt{-3}]$. Recall that the units in this ring are $\{\pm 1,\frac{\pm1 \pm\sqrt{-3}}{2}\}$. 
\[\scalebox{0.85}{\xymatrix{ 
& 1 \ar[rd] &&
\scriptstyle 2-\sqrt{-3} \ar[rd] &&
\frac{3+\sqrt{-3}}{2} \ar[rd] &&
1  &&
\\
\frac{1+\sqrt{-3}}{2} \ar[rd]\ar[ru] &&
\scriptstyle 1-\sqrt{-3} \ar[rd]\ar[ru] &&
\frac{7-\sqrt{-3}}{2} \ar[rd]\ar[ru] &&
\frac{1+\sqrt{-3}}{2} \ar[rd]\ar[ru] &&
 \\ 
& 1 \ar[ru] &&
\scriptstyle 2-\sqrt{-3} \ar[ru] &&
\frac{3+\sqrt{-3}}{2} \ar[ru] &&
1  &&
}}
\] 
In this case, the examples of the diamond relations become $1\cdot (2-\sqrt{-3} )= 1-\sqrt{-3} $ and $ (1-\sqrt{-3})(\frac{7-\sqrt{-3}}{2}   ) =(2-\sqrt{-3} ) ^2+1$.

\subsection{Positive unitary integral friezes}\label{sect 2.4} In this subsection we show that for a positive unitary integral frieze, the cluster that carries the unitarity property is unique.
\begin{proposition}
 \label{prop 4}
 Let $\calf\colon\cala(Q)\to \mathbb{Z}$ be a positive unitary integral frieze and let $\xx$ be a cluster such that $\calf(\xx)=(1,\ldots,1)$. Then for all cluster variables $u\notin \xx$ we have $\calf(u)>1$. In particular $\xx$ is the unique cluster such that $\calf(\xx)=(1,\ldots,1)$.
\end{proposition}
\begin{proof} Suppose $\calf(u)=1$.
 Since $u$ is a Laurent polynomial in $\xx$ with positive coefficients due to  \cite{LS4}, this implies that $u$ is a Laurent monomial in $\xx$.  
 By \cite[Lemma 3.7]{CKLP}, it follows  that $u$ is in $\mathbf{x}$. 
\end{proof}

\section{Frieze vectors}\label{sect 3}
In this section, we introduce a class of positive integer vectors and show that they are in bijection with the clusters of the cluster algebra. 

\subsection{Definition}\label{sect 3.1}
We start with a general result on non-zero friezes.
\begin{proposition}\label{prop 1}
 Every non-zero frieze $\calf\colon \cala(Q)\to R$ is completely determined by its values on an arbitrary cluster in $\cala(Q)$.
\end{proposition}
\begin{proof}
 Let $\xx=(x_1,\cdots,x_n)$ be a cluster in $\cala(Q)$ and let $u$ be an arbitrary cluster variable in $\cala(Q)$ that does not lie in $\xx$. By the Laurent phenomenon \cite{FZ}, we can write $u$ as a Laurent polynomial in $x_1,\ldots,x_n$, thus 
 \[ u= \frac{f(x_1,\ldots,x_n)}{x_1^{d_1}\cdots x_n^{d_n}} \qquad \textup{with } f\in \mathbb{Z}[x_1,\ldots,x_n], d_i\ge 0.\]
Thus \[\calf(u) = \frac{f(\calf(x_1),\ldots,\calf(x_n))}{\calf(x_1)^{d_1}\cdots \calf(x_n)^{d_n}}\]
in the field of fractions of $R$. Note that this expression is well-defined since the frieze is non-zero. Therefore  $\calf(u)$ is determined by the values $\calf(x_i)$. Since the cluster algebra is generated by its cluster variables, this completes the proof.
\end{proof}

Proposition \ref{prop 1} implies that given an arbitrary cluster $\xx=(x_1,\ldots,x_n)$ we can obtain \emph{every} non-zero frieze by specializing the cluster variables $x_i$ of the cluster to certain ring elements $\calf(x_i)=a_i\in R$.
It is important to note that by far not every choice of elements $a_i\in R$ will produce a frieze with values in $R$, because  in general the values will be in the field of fractions of $R$. It is natural to ask which choices $a_i\in R$ do. This leads us to the following definition.

\begin{definition} Let $\xx=(x_1,\ldots,x_n)$ be a cluster of $\cala(Q)$.

 (1) A vector $(a_1,\ldots,a_n) \in R_{\neq0}^n$ is called a \emph{frieze vector relative to $\xx$} if the frieze $\calf$ defined by $\calf(x_i)=a_i$ has values in $R$. If the frieze $\calf$ is unitary we say that the frieze vector  $(a_1,\ldots,a_n)$ is \emph{unitary}.
 
 (2) A vector $(a_1,\ldots,a_n) \in \mathbb{Z}_{>0}^n$ is called a \emph{positive frieze vector relative to $\xx$} if the frieze $\calf$ defined by $\calf(x_i)=a_i$ is positive integral.
\end{definition}

\begin{proposition}
 \label{prop 2}
{\rm (1)} Let $(a_1,\ldots,a_n)\in R^n$ such that every $a_i$ is a unit in $R$. Then $(a_1,\ldots,a_n)$ is a (unitary) frieze vector relative to every cluster $\xx=(x_1,\ldots,x_n)$ in $\cala(Q)$.

{\rm (2)} The vector $(1,\ldots,1)\in \mathbb{Z}_{>0}^n$  is a positive (unitary) frieze vector relative to every cluster $\xx=(x_1,\ldots,x_n)$ in $\cala(Q)$. 
\end{proposition}

\begin{proof}
 (1) By the Laurent phenomenon, every cluster variable is a Laurent polynomial in $\xx$. Since each $x_i$ is specialized to a unit in $R$, the denominator of this Laurent polynomial also specializes to  a unit in $R$. Therefore the image of every cluster variable lies in $R$, and hence $\calf(\cala(Q))\subset R$. 
 
 (2) The frieze is integral by part (1) and positivity follows from the positivity theorem for cluster variables \cite{LS4}.
\end{proof}

\subsection{Acyclic type}
In the case where the quiver $Q$ is mutation equivalent to an acyclic quiver, we have the following characterization of frieze vectors. 
\begin{proposition}
 \label{prop 3}
 Let $(\xx=(x_1,\ldots,x_n),Q)$ be an acyclic seed of the cluster algebra. Then a vector $(a_1,\ldots,a_n)\in R^n$ is a frieze vector relative to $\xx$ if and only if $\textup{$a_i$ is a divisor of } \prod_{i\to j} a_j +\prod_{i\leftarrow j} a_j$
 in $R$, for all $i=1,\ldots,n$.
\end{proposition}
\begin{proof}
 Let $x_i'$ denote the cluster variable obtained from $(\xx,Q)$ by mutating in direction $i$. Then 
 
\begin{equation}\label{eq xi'}
 x_i'= \frac{\prod_{i\to j} x_j +\prod_{i\leftarrow j} x_j}{x_i}.
\end{equation}
 By \cite[Corollary 1.21]{BFZ}, the cluster algebra is generated by the $2n$ variables $x_1,\ldots,x_n,x_1',\ldots,x_n'$.  
  Let $\calf $ be the homomorphism defined by $\calf(x_i)=a_i$. Then 
 \[\begin{array}{ll}&\calf(\cala(Q))\subset R \\ \ssi& \calf(x_i')\in R \textup{ for each $i$ }\\ 
 \ssi&  a_i \textup{ divides }  \prod_{i\to j} a_j +\prod_{i\leftarrow j} a_j \textup{ in $R$ for all $i$.}\end{array}\]
\end{proof}

The following is a special case of Proposition~\ref{prop 3}.

\begin{example}[{\cite[Problem 24]{ConCox1,ConCox2}}]
Suppose $Q$ is the linearly-oriented type $\mathbb{A}$ quiver $1 \to 2 \to \dots \to n$. 
Then the vector $(a_1,\dots, a_n)\in \mathbb{Z}_{>0}^n$ 
is a (unitary, integral) frieze vector relative to a seed with quiver $Q$ if and only if 
the entry $a_1$ divides $1+a_2$, 
the entry
$a_n$ divides $a_{n-1}+1$, and 
the entry
$a_i$ divides $a_{i-1}+a_{i+1}$ for all $1<i<n$.
\end{example}
\begin{remark}
 In the special case of the linearly oriented type $\mathbb{A}$ quiver, the proposition shows that the frieze vectors are related to arithmetical structures on the path graph \cite{Braun+}.
\end{remark}

\subsection{Main result on frieze vectors}
We are now ready to state and prove our first main result. 

\begin{theorem}
 \label{thm1}
 Let $Q$ be a quiver without loops and 2-cycles and let $\xx=(x_1,\ldots,x_n)$ be an arbitrary cluster of $\cala(Q)$. Then there is a bijection
 \[
\begin{array}{rcl}
 \phi\colon\{\textup{unordered clusters in $\cala(Q)$}\} &\longrightarrow&\left\{\begin{array}{l}\textup{positive unitary frieze } \\ \textup{vectors relative to $\xx$}\end{array}\right\}
 \\
 \xx'=\{x_1',\ldots,x_n'\} &\longmapsto& \phi(\xx')=(a_1,\ldots,a_n).
\end{array}\]
\end{theorem}

\begin{remark}
 (1) The theorem implies that every cluster $\xx'$ defines a unique positive unitary frieze vector in $\mathbb{Z}_{>0}^n$. This vector is different from the $g$-vector and the $c$-vector of the seed. 
 
 (2) We stress that, while the order of the cluster variables $x_1',\ldots,x_n'$ is irrelevant, the order of the entries of the frieze vector $\phi(\xx')=(a_1,\ldots,a_n)$ is important. In other words, if $\zs$ is a permutation then $\phi(\zs \xx')=\phi(\xx')$, but $\zs\phi(\xx')\ne \phi(\xx')$ in general.
\end{remark}
\begin{proof}
 Each cluster variable $x_1,\ldots,x_n$ in the fixed cluster $\xx$ can be expressed as a Laurent polynomial in the cluster $\xx'$, say $x_i=\call_i(x_1',\ldots,x_n')$. We define the map $\phi$ by $\phi(\xx')=(a_1,\ldots,a_n) $, with $a_i=\call_i(1,\ldots,1)$. In other words, $\phi(\xx')$ is equal to the vector $\calf(\xx)=(a_1,\ldots,a_n)$, where $\calf$ is the frieze defined by specializing the cluster variables in $\xx'$ to 1. By Proposition \ref{prop 2}, the frieze $\calf$ is unitary, integral and positive. Thus $(a_1,\ldots,a_n)$ is a positive unitary frieze vector relative to $\xx$.  Furthermore, since every variable in $\xx'$ is specialized to 1,  we clearly have $\phi(\zs \xx')=\phi(\xx')$, for every permutation $\zs$. Thus the map $\phi $ is well-defined. 
 
 To show that $\phi$ is surjective, let $(a_1,\ldots,a_n)\in \mathbb{Z}_{>0}^n$ be any positive unitary frieze vector relative to $\xx$.  By definition, the corresponding frieze defined by $\calf(x_i)=a_i$ is positive and unitary, which means that there exists a cluster $\xx'=(x_1',\ldots,x_n')$ such that $\calf(x_i')=1$, for  $i=1,\ldots,n$.  By construction of $\phi$, we have $\phi(\xx')=(a_1,\ldots,a_n)$, so $\phi $ is surjective.
 
 To show injectivity, let $\xx',\xx''$ be two clusters in $\cala(Q)$ such that $\phi(\xx')=\phi(\xx'')$. Let $\calf'$ and $\calf''$ be the unitary friezes defined by $\calf'(x_i')=1$ and $\calf''(x_i'')=1$, respectively. Since $\phi(\xx')=\phi(\xx'')$, both friezes have the same values on $\xx$, thus $\calf'(\xx)=\calf''(\xx)=(a_1,\ldots,a_n)$. Now Proposition \ref{prop 1} implies that $\calf'=\calf''$, and Proposition \ref{prop 4} yields $\xx'=\xx''$.
\end{proof}

\begin{remark}
 The inverse of the bijection $\phi$ is given as follows.  Given a positive unitary frieze vector $(a_1,\dots, a_n)$, we compute the corresponding unitary frieze $\calf$ by specializing $(x_1,\dots,x_n)=(a_1,\dots,a_n)$. By Proposition \ref{prop 4}, this frieze has a unique cluster  $\mathbf{x'}$ such that $\calf(\xx')=(1,\ldots,1)$.  Then $\phi^{-1}(a_1,\ldots,a_n)=\xx'$.
\end{remark}

 \subsection{Example} Thanks to Proposition \ref{prop 3}, the positive integral frieze vectors $(a_1,a_2,a_3)$ relative to the seed $(x_1,x_2,x_3),1\to 2 \leftarrow 3$ are characterized by the condition that the following three expressions are integers 
 \[\frac{a_2+1}{a_1},\   \frac{a_1a_3+1}{a_2},\ \frac{a_2+1}{a_3}.\] 
 The 14 frieze vectors $(a_1,a_2,a_3)$ are the following.
 \[
\begin{array}
 {cccccccccccccc} (1,1,1)&(1,1,2)&(1,2,1)&(1,2,3)&(1,3,2) & (2,1,1) & (2,1,2) & (2,3,1)\\(2,3,4)&(2,5,2) & (3,2,1) &(3,2,3)&(3,5,3)&(4,3,2)
\end{array}
 \]
 Equivalently, we can think of the conditions as Diophantine equations in two sets of integers as follows.
  \[a_1b_1=a_2+1,\  a_2b_2=a_1a_3+1,\  a_3b_3={a_2+1}.\] 
  Note that $b_i$ is the number of terms in the cluster variable $x_i'$ obtained from the initial cluster by mutation in direction $i$, see Equation (\ref{eq xi'}).
The vectors $(b_1,b_2,b_3)$, in the same order as the frieze vectors above, are the following.
\[\begin{array}
 {cccccccccccccc} (2,2,2)&(2,3,1)&(3,1,3)&(3,2,1)&(4,1,2) & (1,3,2) & (1,5,1) & (2,1,4)\\(2,3,1)&(3,1,3) & (1,2,3) &(1,5,1)&(2,2,2)&(1,3,2)
\end{array} \]
  Figure \ref{fig 3} shows the frieze vectors and their clusters in the exchange graph, where the clusters are illustrated by their position in the Auslander-Reiten quiver of the cluster category.
 
\begin{figure}[hbpt]
\Large\scalebox{0.5}{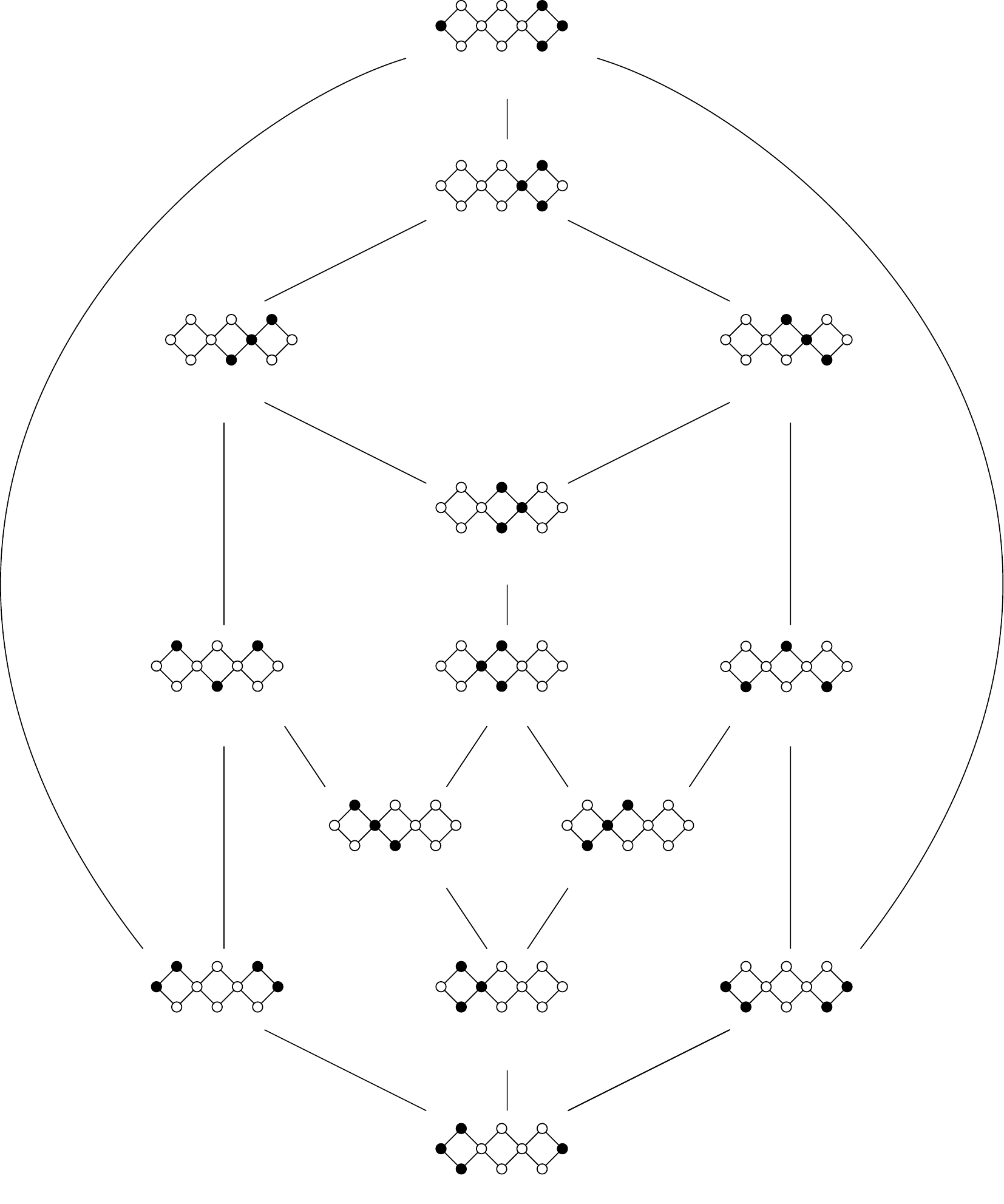}
 \caption{Frieze vectors relative to   $(x_1,x_2,x_3),1\to 2 \leftarrow 3$ together with their clusters.}
 \label{fig 3}
\end{figure}

\subsection{Mutation of frieze vectors in type $\mathbb{A}$} 
Mutations of positive integral friezes are described in  \cite{BFGST, BFGSTsurvey}, where the authors compute the effect of mutation on the whole frieze. Here, we are interested in describing the effect of mutation on the frieze vector relative to a fixed cluster $\xx$. To give this description, we use the combinatorial formula of \cite{MS} to write the cluster variables of $\xx$ with respect to the cluster $\xx'$ in terms of perfect matchings of snake graphs. Then the values in the frieze vectors are simply given as the number of perfect matchings of the appropriate snake graph. 

We will not define snake graphs here but rather refer to the survey \cite{S2}. For our purpose  it suffices to say that a snake graph is a planar graph consisting of a sequence of square tiles that are glued together such that two consecutive tiles share exactly one edge which is either the north edge of the first tile and the south edge of the second tile or 
 the east edge of the first tile and the west edge of the second tile. We associate a snake graph to each cluster variable in $\xx$. The tiles  of the snake graph are labeled by the cluster variables in the cluster $\xx'=(x_1',\ldots,x_n')$ and its edges are labeled by the cluster variables in $\xx'$ or by the constant 1. Since our cluster algebra is of Dynkin type $\mathbb{A}$, no two tiles have the same label and no two interior edges are labeled by the same cluster variable. 
 
 The mutation from $\xx'$ to $\xx''=(\xx'\setminus\{x_i'\})\cup \{x_i''\} $ has the following effect on the snake graphs from a cluster algebra of Dynkin type $\mathbb{A}$. 
 
\begin{enumerate}
\item If the first or last tile of the snake graph has label $x_i'$ then this tile is removed  and the new boundary edge is labeled by the new cluster variable $x_i''$, see the top row of Figure \ref{fig 5}.
Conversely, if the snake graph ends with an edge that is labeled $x_i'$ then a new tile with label $x_i''$ is glued to this edge. 
\item If the snake graph has a tile labeled $x_i'$ that is the middle tile of a 3-tile straight subgraph then  it transforms as shown in the second row  of Figure \ref{fig 5}.
\item If the snake graph has a tile labeled $x_i'$ that is the middle tile of a 3-tile  subgraph that is not straight, then  it transforms as shown in the third row  of Figure \ref{fig 5}.
Conversely, if the snake graph contains an interior edge labeled $x_i'$ shared by two tiles with labels $x_h',x_j'$ then a new tile labeled $x_i''$ is inserted such that the three consecutive tiles labeled $x_h',x_i'',x_j'$ do not form a straight subgraph. 
\end{enumerate}

\begin{figure}[hbpt]
\scalebox{0.8}{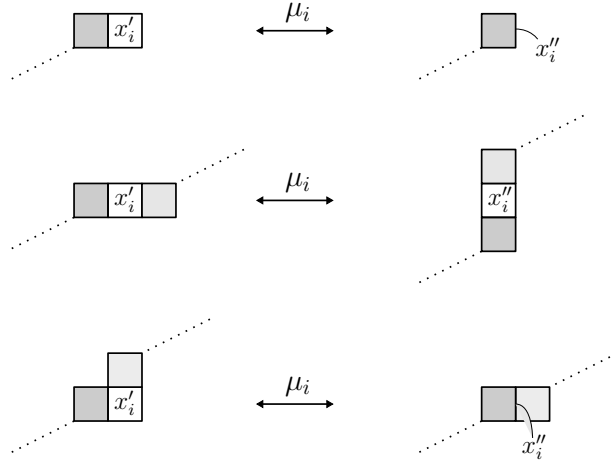}
 \caption{Mutation of snake graphs in direction $x_i'$.}
 \label{fig 5}
\end{figure}

\smallskip
The above description gives the mutations of frieze vectors in the example of Figure \ref{fig 3}. 
For example the mutations   
of frieze vectors in Figure \ref{fig frieze mutation}
are given by the snake graph mutations in Figure~\ref{fig snake graph mutation}. 

\begin{figure}[hbpt]
\begin{tikzpicture}[xscale=1.2, yscale=1.2]
\node (353) at (0,0) {$(3,5,3)$};
\node (234) at (-1,1) {$(2,3,4)$};
\node (432) at (1,1) {$(4,3,2)$};
\node (252) at (0,-1) {$(2,5,2)$};
\path[->, thick] (353) edge[] (252);
\path[->, thick] (353) edge[] (234);
\path[->, thick] (353) edge[] (432);
\end{tikzpicture} 
\caption{Mutations of frieze vectors.}\label{fig frieze mutation}
\end{figure}

\begin{figure}[hbpt]
\centerline{ 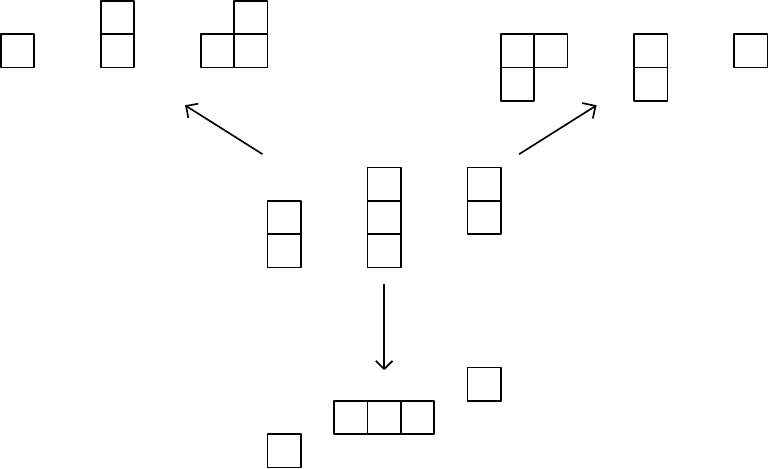}
\caption{Snake graph mutations of the frieze vectors in Figure \ref{fig frieze mutation}.}\label{fig snake graph mutation}
\end{figure}

\section{Friezes of type $\widetilde{\mathbb{A}}$}\label{sect 4}
In this section, we study the special case of integral friezes of affine Dynkin type $\mathbb{A}$. We show that every positive integral frieze of this type is unitary.

Let $Q$ be a quiver that is mutation equivalent to a quiver $Q'$ of type $\widetilde{\mathbb{A}}_{p,q}$. The cluster algebra $\cala(Q)$ is of surface type and the corresponding surface is an annulus with $p$ marked points on one boundary component and $q$ marked points on the other boundary component, see \cite{FST}. The cluster variables $x_\zg$ in $\cala(Q)$ are in bijection with the arcs $\zg$ in the annulus. We call a cluster variable $x_\zg$ \emph{transjective} if its arc $\zg$ has its two endpoints on two different boundary components (bridging arc) and we call the cluster variable $x_\zg$ \emph{regular} if the arc $\zg$ has both endpoints on the same boundary component (peripheral arc). The terminology transjective versus regular comes from the cluster category $\calc_Q$.

\begin{lemma}
 \label{lem 4.1}
 Let $\calf\colon\cala(Q)\to \ZZ$ be a positive integral frieze of type $\widetilde{\mathbb{A}}_{p,q}$ and let $\xx=(x_1,\ldots,x_n)$ be a cluster such that $\calf(x)=1$ for each regular cluster variable $x\in \xx$ if any. Let $i$ be such that $\calf(x_i)\ge\calf(x_j)$ for all $j$, and suppose that $\calf(x_i)>1$. Let $x_i'$ be the cluster variable obtained from $\xx$ by mutation in direction $i$. Then $\calf(x_i')<\calf(x_i)$ and if $x_i'$ is a regular cluster variable then $\calf(x_i')=1.$
\end{lemma}
\begin{proof}
 Let $\tau_j$ be the arc corresponding to the cluster variable $x_j$, so that $T=(\tau_1,\ldots,\tau_n)$ is the triangulation corresponding to the cluster $\xx$. The mutation in direction $i$ is given by flipping the arc $\tau_i$ in $T$, and the exchange relation in the cluster algebra is of the form
 \begin{equation}\label{eq 1} x_ix_i'=x_ax_c+x_bx_d\end{equation}
 where $\tau_i$ is the diagonal in the quadrilateral in $T$ with sides $\tau_a,\tau_b,\tau_c,\tau_d$ as in Figure \ref{fig 1} some of which may be boundary edges.
 
\begin{figure}
\scalebox{0.8}{ 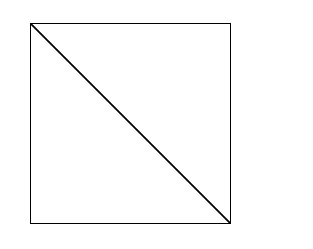}
 \caption{Quadrilateral in the triangulation $T$.}
 \label{fig 1}
\end{figure}

Our assumption that $\calf(x_i)>1$ and $\calf(x)=1$ for every regular cluster variable $x\in \xx$ imply that $x_i$ is transjective. Hence $\tau_i$ is a bridging arc, so its endpoints lie on different boundary components. Therefore one of the arcs $\tau_a,\tau_b$ is bridging and the other is peripheral (or a boundary edge), and also one of  $\tau_c,\tau_d$ is bridging and the other is peripheral (or a boundary edge). We assume without loss of generality that $\tau_a$ is bridging and consider two cases. 

Suppose first that $\tau_c$ is bridging. Then the relation (\ref{eq 1}) implies
\begin{equation}
 \label{eq 2}
 \calf(x_i')=(\calf(x_a)\calf(x_c)+1)/\calf(x_i)
\end{equation}
because the frieze has value 1 on the two regular variables (or boundary edge weights) $x_b$ and $x_d$. Note that in this case the flipped arc $\tau'_i$ is bridging. Recall that $\calf(x_a)\le\calf(x_i)$ and $\calf(x_c)\le \calf(x_i)$. If $\calf(x_a)=\calf(x_i)$ then the right hand side of (\ref{eq 2}) would be equal to $\calf(x_c)+ 1/\calf(x_i)$ which is not an integer. Thus $\calf(x_a)<\calf(x_i)$ and similarly $\calf(x_c)<\calf(x_i)$. Therefore the right hand side of (\ref{eq 2}) is at most $((\calf(x_i)-1)^2+1)/\calf(x_i)=\calf(x_i)-2+(2/\calf(x_i))$ which is strictly smaller than $\calf(x_i)$, and we are done.

Suppose now that $\tau_c$ is a peripheral arc. Then $\tau_d$ is bridging and the relation (\ref{eq 1}) implies
\begin{equation}
 \label{eq 3}
 \calf(x_i')=(\calf(x_a)+\calf(x_d))/\calf(x_i)
\end{equation}
Note that in this case the arc $\tau_i'$ is peripheral and forms a triangle with the two peripheral arcs $\tau_b$ and $\tau_c$. We will show that $\calf(x_i')=1$. Since $\calf(x_i)$ is the maximal frieze value in $\xx$, equation (\ref{eq 3}) yields 
$ \calf(x_i')\le 2\calf(x_i)/\calf(x_i)=2.
$
If $\calf(x_i')=1$ we are done. Assume therefore that $\calf(x_i')=2$. 
Then equation (\ref{eq 3}) implies
\begin{equation}
 \label{eq 4}
 \calf(x_a)=\calf(x_d)=\calf(x_i)\ge 2.
\end{equation}
Consider the quadrilateral in $T$ in which $\tau_d$ is the diagonal and denote its sides $\tau_i,\tau_c,\tau_e,\tau_f$ where $\tau_i,\tau_e $ are bridging arcs and $\tau_c,\tau_f$ are peripheral, see Figure \ref{fig 2}.

\begin{figure}
\scalebox{0.7}{ 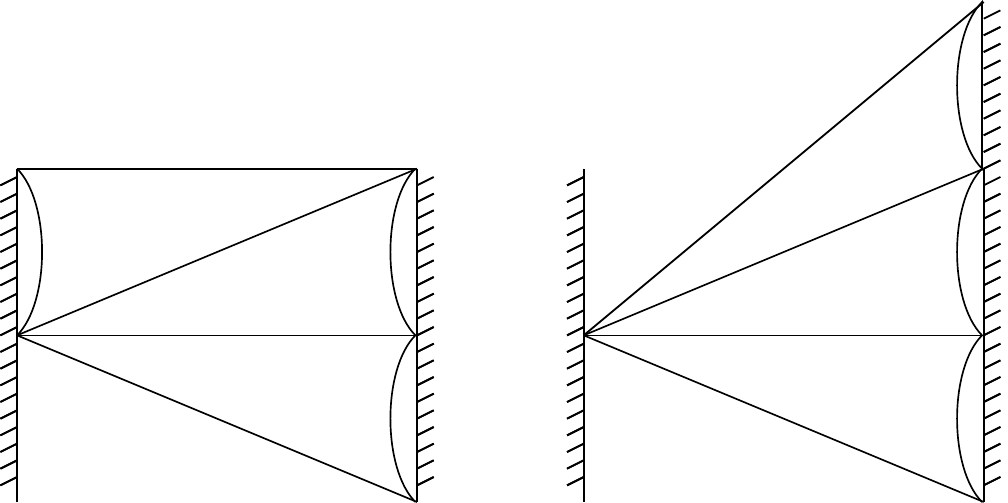}
 \caption{Two possible configurations in the triangulation $T$ when $\tau_c$ is  a peripheral arc or a boundary edge.}
 \label{fig 2}
\end{figure}

Let $x_d'$ be the cluster variable obtained by mutating $\xx$ in direction $d$. Then in the situation of the left picture in Figure \ref{fig 2} we have
\[\calf(x_d')=(\calf(x_i)\calf(x_e)+1)/\calf(x_d) =\calf(x_e) +1/\calf(x_i),\] 
where the last equality holds by (\ref{eq 4}). But since $\calf(x_i)\ge 2$, this expression is not an integer, so we have a contradiction. 

Therefore we must be in the situation of the right picture in Figure \ref{fig 2}, and we have
\[\calf(x_d')=(\calf(x_i)+\calf(x_e))/\calf(x_d) =1+\calf(x_e)/\calf(x_i),\]
where the last identity holds by (\ref{eq 4}). Since $\calf(x_i)\ge\calf(x_e)$ and $\calf$ is  a positive integral frieze, we must have $\calf(x_i)=\calf(x_e)$ and $\calf(x_d')=2$.

 We have thus shown that if $\calf(x_i')=2$ then the triangulation $T$ contains a fan of bridging arcs $\tau_i,\tau_d,\tau_e$ and $\calf(x_d')=2, \calf(x_e)=\calf(x_i)$.
We can now repeat this argument by considering the cluster variable $x_e'$ obtained by mutating $\xx$ in direction $e$, and recursively with every new bridging arc in the fan and we obtain a fan of bridging arcs in $T$ and each arc in this fan has the same frieze value $\calf(x_i)\ge 2$. Since $T$ is a triangulation of the annulus, this fan is finite, and the two arcs bounding it correspond to a sink and a source in the quiver $Q_T$. Mutating at one of those arcs  again gives a contradiction as in the left picture of Figure \ref{fig 2}. We have shown that $\calf(x_i')$ cannot be equal to 2, and thus $\calf(x_i')=1$. 
\end{proof}

We are now ready for the main theorem of this section.
\begin{theorem}
\label{thm 2} Let $Q$ be a quiver of type $\widetilde{\mathbb{A}}_{p,q}$ and let $\calf\colon\cala(Q)\to \ZZ$ be a positive integral frieze. Then $\calf$ is unitary.
\end{theorem}
\begin{proof}
We need to show that there exists a cluster $\xx'$ such that $\calf(\xx')=(1,\ldots,1)$. Let $\xx_0$ be a cluster consisting entirely of transjective cluster variables. Its triangulation $T_0$ consists entirely of bridging arcs. Then $\xx_0=(x_1,\ldots,x_n)$ is a cluster that satisfies the condition of Lemma \ref{lem 4.1}. If $\calf(\xx_0)=(1,\ldots,1)$ we are done. Otherwise Lemma \ref{lem 4.1} implies that mutating at a cluster variable $x_i$ with maximal frieze value will produce a cluster $\xx_1=(\xx_0\setminus\{x_i\})\cup\{x_i'\}$ such that $\calf(x_i')<\calf(x_i)$ and if $x_i'$ is regular then $\calf(x_i')=1$. Therefore, if $\calf(\xx_1)\ne(1,\ldots,1)$ then the cluster $\xx_1$ also satisfies the hypothesis of Lemma~\ref{lem 4.1}, and we can repeat this procedure to produce a sequence of clusters $\xx_0,\xx_1,\ldots,\xx_s,\ldots$ such that $\xx_s=(\xx_{s-1}\setminus\{x\})\cup\{x'\}$ with $\calf(\xx_s)\ne(1,\ldots,1)$ and $\calf(x')<\calf(x)$. Since the frieze is positive integral this process must stop. Thus there is a cluster $\xx_t$ such that $\calf(\xx_t)=(1,\ldots,1)$. 
\end{proof}

\subsection{Friezes of type $\widetilde{\mathbb{A}}_{2,1}$}
There are precisely two positive integral friezes of type $\widetilde{\mathbb{A}}_{2,1}$ up to symmetry, and they are depicted in Figures \ref{fig:A1_2_acyclic} and  \ref{fig:A1_2_cyclic}. By Theorem \ref{thm 2} both are unitary. 
In the first example, the cluster $\xx$ with $\calf(\xx)=(1,1,1)$ is transjective and in the second example one of the cluster variables in $\xx$ is regular. In the figures, we show the values of the friezes on the transjective component of the Auslander-Reiten quiver.

\begin{figure}[h]
\tiny\def\myxscale{0.60}

\begin{tikzpicture}[xscale=\myxscale] 
\node at (-7.3,-1) {$\dots$};
\draw node at (12.5,-1) {$\dots$};

\def\myshift{0.5}

\foreach \n in {-2,...,5}
{
  \foreach \vertex in
  {0,1,2}
  {
    \path[black] (\n*\myshift-\vertex+2*\n,-\vertex) node (x\vertex\n) {};
  }
  \foreach \source/\target in 
  {2/1,1/0}
  {
    \path[->,>=stealth] (x\source\n) edge[blue] (x\target\n);
  }
  \foreach \source/\target in 
  {2/0}
  {
    \path[->,>=stealth] (x\source\n) edge[blue, bend left=55] (x\target\n);
  }  
}

\foreach \nminusone/\n in 
{
-2/-1,-1/0,0/1,1/2,2/3,3/4,4/5}
{
  \foreach \s/\t in 
  {1/2,0/2,0/1}
  {
    \path[->,>=stealth] (x\s\nminusone) edge[red] (x\t\n);
  } 
}  

\foreach \vertex/\n/\weight in
  {
0/-2/11,1/-2/26,2/-2/41,
0/-1/2,1/-1/3,2/-1/7,
0/0/1,1/0/1,2/0/1,
0/1/7,1/1/3,2/1/2,
0/2/41,1/2/26,2/2/11,
0/3/\ \ 362,1/3/153,2/3/97,
0/4/\ \ \ 2131,1/4/1351,2/4/571,
0/5/\ \ \ \ \ 18817,1/5/7953,2/5/5042
  }
  {
    \path[black] (\n*\myshift-\vertex+2*\n,-\vertex) node (x\vertex\n) {\weight};
  }
\end{tikzpicture}  
\caption{An $\widetilde{A}_{1,2}$ frieze obtained by specializing the cluster variables of an acyclic seed to $1$. The two peripheral arcs have  frieze values $2$ and $3$.}\label{fig:A1_2_acyclic}.\

\begin{tikzpicture}[xscale=\myxscale] 
\node at (-7.3,-1) {$\dots$};
\draw node at (12.5,-1) {$\dots$};

\def\myshift{0.5}

\foreach \n in {-2,...,5}
{
  \foreach \vertex in
  {0,1,2}
  {
    \path[black] (\n*\myshift-\vertex+2*\n,-\vertex) node (x\vertex\n) {};
  }
  \foreach \source/\target in 
  {2/1,1/0}
  {
    \path[->,>=stealth] (x\source\n) edge[blue] (x\target\n);
  }
  \foreach \source/\target in 
  {2/0}
  {
    \path[->,>=stealth] (x\source\n) edge[blue, bend left=55] (x\target\n);
  }  
}

\foreach \nminusone/\n in 
{
-2/-1,-1/0,0/1,1/2,2/3,3/4,4/5}
{
  \foreach \s/\t in 
  {1/2,0/2,0/1}
  {
    \path[->,>=stealth] (x\s\nminusone) edge[red] (x\t\n);
  } 
}  

\foreach \vertex/\n/\weight in
  {
0/-2/5,1/-2/18,2/-2/13,
0/-1/3,1/-1/2,2/-1/7,
0/0/1,1/0/2,2/0/1,
0/1/7,1/1/2,2/1/3,
0/2/13,1/2/18,2/2/5,
0/3/\ \ 123,1/3/34,2/3/47,
0/4/\ \ 233,1/4/322,2/4/89,
0/5/\ \ \ \ 2207,1/5/610,2/5/843
  }
  {
    \path[black] (\n*\myshift-\vertex+2*\n,-\vertex) node (x\vertex\n) {\weight};
  }

\end{tikzpicture}  
\caption{An $\widetilde{A}_{1,2}$ frieze obtained by specializing the cluster variables of a non-acyclic seed to $1$. The two peripheral arcs have frieze values $1$ and $5$.}
\label{fig:A1_2_cyclic}

\end{figure}


\subsection{Further unitarity questions} 
It was shown in \cite{ConCox1,ConCox2} that every positive integral frieze of Dynkin type $\mathbb{A}_n$ is unitary, and by Theorem \ref{thm 2}, the same is true for affine type $\widetilde{\mathbb{A}}_{p,q}$. It is natural to ask if these results can be extended to friezes with values in other integral domains, for example in quadratic integer rings. However the following example shows that the result  already fails over the Gaussian integers.

\begin{example}
 Let $Q$ be the quiver $1\to 2$ and define a frieze $\calf\colon\cala(Q)\to\mathbb{Z}[i]$ by $\calf(x_1)=1$ and $\calf(x_2)=1+i$. We can visualize $\calf$ as usual in the Auslander-Reiten quiver as follows
 \[ \xymatrix{&&1+i\ar[rd]&&2-i\ar[rd]&&1\ar[rd]\\
 &1\ar[ru]&&2+i\ar[ru]&&1-i\ar[ru] &&1+i}
 \]
 This is a non-unitary frieze of Dynkin type $\mathbb{A}_2$. 
 We don't know if there exists a non-unitary frieze whose entries are ``positive'' in the sense that they are of the form $a+bi$ with $a,b\ge0$.
\end{example}

\subsubsection{Other Dynkin or affine types} For Dynkin types $\mathbb{D} $ and $\mathbb{E}$ there are non-unitary positive integral friezes, see \cite{FP}, and these examples also give rise to non-unitary positive integral friezes in the affine types $\widetilde{\mathbb{D}}$ and $\widetilde{\mathbb{E}}$.

\subsection*{Acknowledgements}
We thank A. Garcia Elsener, G. Musiker and P.-G. Plamondon for helpful discussions. 
We also thank the anonymous referees for valuable feedback and suggestions. 

\printbibliography

\end{document}